\documentclass[12pt,reqno]{amsart}
\usepackage{amssymb}

\newtheorem{theorem}{Theorem}[section]
\newtheorem{lemma}[theorem]{Lemma}
\newtheorem{proposition}[theorem]{Proposition}

\theoremstyle{definition}

\newtheorem{question}[theorem]{Question}

\numberwithin{equation}{section}

\begin{document}

\baselineskip=15pt

\title[Moduli space of $G$--connections on an elliptic curve]{Moduli space of
$G$--connections on an elliptic curve}

\author[I. Biswas]{Indranil Biswas}

\address{School of Mathematics, Tata Institute of Fundamental
Research, Homi Bhabha Road, Bombay 400005, India}

\email{indranil@math.tifr.res.in}



\thanks{The author is supported by J. C. Bose Fellowship.}

\date{}

\begin{abstract}
Let $X$ be a smooth complex elliptic curve and $G$ a connected reductive affine
algebraic group defined over $\mathbb C$. Let ${\mathcal M}_X(G)$ denote the moduli
space of topologically trivial algebraic $G$--connections on
$X$, that is, pairs of the form $(E_G\, , D)$, where $E_G$ is a topologically trivial
algebraic principal $G$--bundle on $X$, and $D$ is an algebraic connection on $E_G$.
We prove that ${\mathcal M}_X(G)$ does not admit any nonconstant algebraic function
while being biholomorphic to an affine variety.
\end{abstract}

\maketitle

\section{Introduction}

Take an irreducible smooth complex projective curve $X$ of genus one and also
take a connected
reductive complex affine algebraic group $G$. The moduli space of topologically trivial
algebraic principal $G$--bundles on $X$ is extensive studied (see \cite{At1},
\cite{FM}, \cite{FMW}, \cite{La}, \cite{FGN} and references therein). In particular,
an explicit description of this moduli space was obtained. Our aim here is to study the
moduli space of algebraic $G$--connections on $X$. More precisely, let
${\mathcal M}_X(G)$ be the moduli space pairs of the form $(E_G\, ,D)$, where
\begin{itemize}
\item $E_G$ is an algebraic principal $G$--bundle on $X$ such that the underlying
topological principal $G$--bundle is trivial, and

\item $D$ is an algebraic connection on $E_G$.
\end{itemize}
(The definition of an algebraic connection on $E_G$ is recalled in Section
\ref{se2.2}.)
The Riemann--Hilbert correspondence, which sends any flat connection on $X$ to its
monodromy representation, produces a biholomorphism from ${\mathcal M}_X(G)$ to the
Betti moduli space $M^B_X(G)$ that parametrizes all equivalence classes of
homomorphisms from $\pi_1(X)$ to $G$ that lie in the connected component of the trivial
homomorphism (see Section \ref{se4.1}). The moduli space $M^B_X(G)$ is an affine
variety. Since ${\mathcal M}_X(G)$ is biholomorphic to $M^B_X(G)$, we conclude that
there are many nonconstant holomorphic functions on ${\mathcal M}_X(G)$. This contrasts
with the following theorem proved here (see Theorem \ref{thm1}):

\begin{theorem}\label{thm0}
The variety ${\mathcal M}_X(G)$ does not admit any nonconstant algebraic function. 
\end{theorem}

Let $C$ be an irreducible smooth complex projective curve and
${\mathcal M}_C(G)$ the moduli space of $G$--connections on $C$ such that
the underlying topological principal $G$--bundle is trivial. As before,
${\mathcal M}_C(G)$ is biholomorphic to an affine variety, namely the
Betti moduli space. Theorem \ref{thm0} raises the following question:

\begin{question}
Does ${\mathcal M}_C(G)$ admit any nonconstant algebraic function?
\end{question}

\section{Algebraic connections on bundles}

\subsection{Connection on vector bundles}

Let $X$ be an irreducible smooth complex projective
curve of genus one. The holomorphic 
cotangent bundle of $X$ will be denoted by $K_X$. We note that $K_X$ is
algebraically trivial.

An \textit{algebraic connection} on an algebraic vector bundle 
$E\,\longrightarrow\, X$ is a holomorphic differential operator
$$
D\, :\, E\, \longrightarrow\, E\otimes K_X
$$
of order one satisfying the Leibniz rule which says that for any locally defined
holomorphic section $s$ of $E$ and any locally defined holomorphic function $f$ on $X$,
\begin{equation}\label{e1}
D(f\cdot s) \,=\, fD(s) + s\otimes df\, .
\end{equation}
Let $\overline{\partial}_E\, :\, E\, \longrightarrow\, E\otimes \Omega^{0,1}$ be 
the Dolbeault operator defining the holomorphic structure on $E$. It is 
straightforward to check that $D$ is an algebraic connection on $E$ if and only if 
$D+\overline{\partial}_E$ is a flat connection on $E$ such that the locally
defined (in analytic topology) flat sections of $E$ are holomorphic.

A rank $r$ {\it connection} on $X$ is a pair $(E\, , D)$, where 
$E$ is an algebraic vector bundle on $X$ of rank $r$, and $D$ is an algebraic 
connection on $E$.

An algebraic vector bundle $V$ on $X$ is called \textit{semistable} if
$$
\frac{\text{degree}(F)}{\text{rank}(F)}\,\leq\, \frac{\text{degree}(V)}{\text{rank}(V)}
$$
for all algebraic subbundles $F\,\subset\, V$ of positive rank.

\begin{lemma}\label{lem1}
Let $(E\, , D)$ be a rank $r$ connection on $X$. Then the 
algebraic vector bundle $E$ is semistable of degree zero.
\end{lemma}

\begin{proof}
Since $E$ admits a flat connection it follows immediately that $\text{degree}(E)\,=\,0$.

Assume that $E$ is not semistable. Let
$$
0\, \not=\, F\,\subsetneq\, E
$$
be the first term of the Harder--Narasimhan filtration of $E$, so the algebraic 
vector bundle $F$ is semistable, and
\begin{equation}\label{e2}
\frac{\text{degree}(F)}{\text{rank}(F)}\, > \frac{\text{degree}(V)}{\text{rank}(V)}
\end{equation}
for every algebraic subbundle $V\,\subset\, E/F$ of positive rank. Now consider 
the second fundamental form $S_D(F)$ for $F$, which is the following composition:
$$
F\,\hookrightarrow\, E\, \stackrel{D}{\longrightarrow}\, E\otimes K_X \,
\longrightarrow\, (E/F)\otimes K_X\, .
$$
{}From \eqref{e1} it follows immediately that $S_D(F)(fs)\,=\, fS_D(F)(s)$, where 
$s$ is any locally defined holomorphic section of $F$ and $f$ is any locally 
defined holomorphic function on $X$. Therefore, $S_D(F)$ is a holomorphic 
homomorphism of vector bundles. Since $K_X$ is trivial, from \eqref{e2} it follows 
that there is no nonzero holomorphic homomorphism of vector bundles from $F$ to 
$E/F$. So we conclude that $S_D(F)\,=\, 0$. Therefore, $D$ induces an algebraic 
connection on $F$. Since $F$ admits an algebraic connection, we have
$$
\text{degree}(F)\, =\, 0\, .
$$
On the other hand, since $\text{degree}(E)\,=\,0$, from \eqref{e2} it follows that
$$
\text{degree}(F)\, > \, 0\, .
$$
In view of the above contradiction, we conclude that $E$ is semistable.
\end{proof}

\begin{lemma}\label{lem2}
Let $E$ be a semistable vector bundle on $X$ of rank $r$ and degree zero.
Then $E$ admits an algebraic connection.
\end{lemma}

\begin{proof}
A theorem due to Atiyah and Weil says that an algebraic vector bundle $W$ over an
irreducible smooth complex projective curve
admits an algebraic connection if and only if 
the degree of any direct summand of $W$ is zero \cite{We}, \cite[p. 202, 
Proposition 17]{At}. Since $E$ is semistable of degree zero, every direct summand 
of $E$ is clearly of degree zero. Therefore, $E$ admits an algebraic connection.
\end{proof}

\subsection{Criterion for principal bundles}\label{se2.2}

Let $G$ be a connected reductive affine algebraic group defined over $\mathbb C$.
The center of $G$ is denoted by $Z(G)$. Let
$$
Z_0(G)\, \subset\, Z(G)
$$
be the connected component containing the identity element.
The Lie algebra of $G$ will be denoted by $\mathfrak g$.

An algebraic principal $G$--bundle $E_G$ over the elliptic curve $X$ is called
\textit{semistable} if for every triple of the form $(P\, ,E_P\, ,\chi)$, where
\begin{itemize}
\item $P\, \subset\, G$ is a parabolic subgroup,

\item $E_P\, \subset\, E_G$ is an algebraic reduction of structure group of $E_G$
to $P$, and

\item $\chi\, :\, P\, \longrightarrow\, {\mathbb C}^*$ is a character such that
$\chi\vert_{Z_0(G)}$ is trivial and the line bundle on $G/P$ associated to $\chi$
is ample,
\end{itemize}
the inequality $\text{degree}(E_P(\chi))\, \geq\, 0$, where $E_P(\chi)\,\longrightarrow
\, X$ is the algebraic line bundle associated to $E_P$ for the character $\chi$.
(See \cite{Ra}, \cite{RS}, \cite{AnBi}.)

For $G\, =\, \text{GL}(r,{\mathbb C})$, the above definition of semistability of $E_G$
is equivalent to the definition of semistability of the vector bundle of rank $r$
associated to $E_G$.

Given an algebraic principal $G$--bundle $p\, :\, E_G\, \longrightarrow\, X$, we have
the Atiyah exact sequence on $X$
$$
0\,\longrightarrow\, \text{ad}(E_G)\,:=\,E_G\times^G{\mathfrak g}
\,\longrightarrow\, \text{At}(E_G)\,:=\,(p_*TE_G)^G \,\longrightarrow\, TX
\,\longrightarrow\,0
$$
(see \cite{At}). An \textit{algebraic connection} on $E_G$ is a holomorphic
(hence algebraic) splitting of the above Atiyah exact sequence.

\begin{proposition}\label{prop1}
An algebraic principal $G$--bundle $E_G$ over $X$ admits an algebraic
connection if and only if the following two conditions hold:
\begin{enumerate}
\item $E_G$ is semistable, and

\item for every character $\chi$ of $G$, the degree of the associated line bundle
$E_G(\chi)\,\longrightarrow\, X$ is zero.
\end{enumerate}
\end{proposition}

\begin{proof}
We first recall a criterion for $E_G$ to admit an algebraic connection. The
principal $G$--bundle $E_G$ admits an algebraic
connection if and only if the following two conditions hold:
\begin{enumerate}
\item The adjoint vector bundle $\text{ad}(E_G)$
admits an algebraic connection, and

\item for every character $\chi$ of $G$, the degree of the associated line bundle
$E_G(\chi)\,\longrightarrow\, X$ is zero.
\end{enumerate}
(See \cite[p. 444, Theorem 3.1]{AzBi}.) On the other hand, the principal $G$--bundle
$E_G$ is semistable if and only the vector bundle $\text{ad}(E_G)$ is semistable
\cite[p. 214, Proposition 2.10]{AnBi}. Since any algebraic connection on $E_G$
induces an algebraic connection on $\text{ad}(E_G)$, the proposition follows by
combining these with Lemma \ref{lem1} and Lemma \ref{lem2}.
\end{proof}

\section{Moduli space of connections}

Let ${\mathcal M}_X$ denote the moduli space of rank one algebraic connections
on $X$, so ${\mathcal M}_X$ parametrizes isomorphism classes of pairs of
the form $(L\, ,D)$, where $L$ is an algebraic line bundle on $X$ of degree zero,
and $D$ is an algebraic connection on $L$. Let
$$
J \,=\, J(X)\,=\, \text{Pic}^0(X)
$$
be the Jacobian of $X$. Once we fix a point $x_0\,\in\, X$, there is the
isomorphism $X\, \longrightarrow\, J$ defined by $x\,\longmapsto\, {\mathcal O}_X
(x-x_0)$. Let
\begin{equation}\label{e3}
\varphi\, :\, {\mathcal M}_X\,\longrightarrow\, J
\end{equation}
be the projection defined by $(L\, ,D)\,\longmapsto\, L$. This $\varphi$ is
a smooth surjective algebraic morphism. The space of all algebraic connections
on an algebraic line bundle $L$ is an affine space for $H^0(X,\, K_X)$.
Therefore, $\varphi$ in \eqref{e3} makes
${\mathcal M}_X$ a torsor over $J$ for the trivial vector bundle $J\times
H^0(X,\, K_X)$. Equivalently, ${\mathcal M}_X$ is an algebraic principal bundle
over $J$ with structure group $H^0(X,\, K_X)$.

Using Serre duality, $H^0(X,\, K_X)$ is identified with $\mathbb C$. Therefore,
${\mathcal M}_X$ is an algebraic principal $\mathbb C$--bundle over $J$.

\begin{lemma}\label{lem3}
The algebraic principal $\mathbb C$--bundle ${\mathcal M}_X\,
\stackrel{\varphi}{\longrightarrow}\, J$ is nontrivial.
\end{lemma}

\begin{proof}
Isomorphism classes of algebraic principal $\mathbb C$--bundles
on $J$ are parametrized by $H^1(J,\, {\mathcal O}_J)$.
Indeed, given a principal $\mathbb C$--bundle $E$ on $J$, choosing local
trivializations (in analytic topology) of $E$ and taking their differences, the
cohomology class in $H^1(J,\, {\mathcal O}_J)$ corresponding to $E$
is constructed. We will now recall the Dolbeault analog of this construction.

Let $\beta\, :\, E\,\longrightarrow\, J$ be an algebraic principal
$\mathbb C$--bundle on $J$. Since $\mathbb C$ is contractible, the
fiber bundle $E$ admits a $C^\infty$ section. Let
$$
\gamma\, :\, J\, \longrightarrow\, E
$$
be a $C^\infty$ section, meaning $\beta\circ\gamma\,=\, \text{Id}_J$. We note
that $\gamma$ need not be holomorphic; in fact, $\beta$ admits a holomorphic
section if and only if the algebraic principal $\mathbb C$--bundle $E$
is trivial. Take a point $x\,\in\, J$. Let
$$
d\gamma (x)\, :\, T^{\mathbb R}_xJ\, \longrightarrow\, T^{\mathbb R}_{\gamma(x)} E
$$
be the differential of $\gamma$ at $x$, where $T^{\mathbb R}$ denotes the
real tangent space. Let ${\mathbb J}_x$ (respectively, ${\mathbb J}_{\gamma(x)}$)
be the almost complex structure on $J$ (respectively, $E$) at $x$
(respectively, $\gamma(x)$). Now consider the homomorphism
$$
T^{\mathbb R}_xJ\, \longrightarrow\, T^{\mathbb R}_{\gamma(x)} E\, , ~\
v\,\longmapsto\, d\gamma (x)({\mathbb J}_x(v))-
{\mathbb J}_{\gamma(x)}(d\gamma (x)(v))\, .
$$
Since the map $\beta$ is holomorphic, this homomorphism produces a
homomorphism
$$
\omega(x)\, :\, T^{0,1}_xJ\, \longrightarrow\, \mathbb C\, ;
$$
using the action of $\mathbb C$ on $E$, the vertical tangent bundle on $E$
for the projection $\beta$ is algebraically identified with the trivial line bundle
with fiber $\mathbb C$ (the Lie algebra of the group $\mathbb C$).
Therefore, we get a $(0\, ,1)$--form $\omega$ on $J$ whose evaluation
at any point $x\, \in\, J$ is $\omega(x)$ defined above. We note that $\omega$ is
the obstruction for the map $\gamma$ to be holomorphic. The element in
$H^1(J,\, {\mathcal O}_J)$ represented by $\omega$ is the cohomology class
corresponding to the algebraic principal $\mathbb C$--bundle $E$.

Now consider the algebraic principal $\mathbb C$--bundle ${\mathcal M}_X$ in
\eqref{e3}. Any algebraic line bundle on $X$ of degree zero admits a unique
unitary flat connection. Let
\begin{equation}\label{ga}
\gamma\, :\, J\, \longrightarrow\, {\mathcal M}_X
\end{equation}
be $C^\infty$ section defined by $L\, \longmapsto\, (L\, , D^u_L)$, where
$D^u_L$ is the unique unitary flat connection on $L$. Consider the
$(0\, ,1)$--form $\omega$ on $J$ constructed as above from this section $\gamma$.
It is straightforward to check that
\begin{itemize}
\item $\omega$ is invariant under the translations of the group $J$, and

\item $\omega$ is nonzero.
\end{itemize}
These two imply that the cohomology class in $H^1(J,\, {\mathcal O}_J)$
represented by $\omega$ is nonzero. This completes the proof.
\end{proof}

\begin{lemma}\label{lem4}
Let $E$ and $F$ be two algebraic principal $\mathbb C$--bundles on $J$ such that
both $E$ and $F$ are nontrivial. Then the total spaces of $E$ and $F$ are
algebraically isomorphic.
\end{lemma}

\begin{proof}
Let $\eta\, :\, E\times {\mathbb C}\,\longrightarrow\, E$ be
the action of $\mathbb C$ on the principal $\mathbb C$--bundle $E$.
Take any nonzero complex number $\lambda$. Let $E^\lambda$ be the
algebraic principal $\mathbb C$--bundle on $J$ whose total space is
$E$ but the action of any $c\, \in\, \mathbb C$ on $E$ is the morphism
$z\,\longmapsto\, \eta(z\, ,\lambda\cdot c)$. If $\omega_0\,\in\,
H^1(J,\, {\mathcal O}_J)$ is the cohomology class corresponding to $E$, then
the cohomology class corresponding to $E^\lambda$ is $\omega_0/\lambda$.
Since $\dim H^1(J,\, {\mathcal O}_J)\,=\, 1$, it now follows that the
total spaces of any two nontrivial algebraic principal $\mathbb C$--bundles on $J$ 
are algebraically isomorphic.
\end{proof}

Let
\begin{equation}\label{V}
0\,\longrightarrow\, {\mathcal O}_J \,\longrightarrow\, V\,
\stackrel{\sigma}{\longrightarrow}\, {\mathcal O}_J \,\longrightarrow\, 0
\end{equation}
be a nontrivial extension of ${\mathcal O}_J$ by ${\mathcal O}_J$; such an
extension exists because $H^1(J,\, {\mathcal O}_J)\,\not=\, 0$. Let
$$
s_1\,:\, J\, \longrightarrow\, {\mathcal O}_J
$$
be the section given by the constant function $1$ on $J$. Define
\begin{equation}\label{Z}
Z\,:=\, \sigma^{-1}(s_1(J))\,\subset\, V\, ,
\end{equation}
where $\sigma$ is the homomorphism in \eqref{V}. Using the inclusion
$$
{\mathcal O}_J \,\hookrightarrow\, Z\,\subset\, V
$$
in \eqref{V} it follows that $Z$ is an algebraic principal $\mathbb C$--bundle over $J$.

\begin{proposition}\label{prop2}
The variety $Z$ in \eqref{Z} is isomorphic to ${\mathcal M}_X$.
\end{proposition}

\begin{proof}
We will show that the algebraic principal $\mathbb C$--bundle $Z\,\longrightarrow\,
J$ is nontrivial. Since local trivializations of $Z$ give local splittings of
the exact sequence in \eqref{V}, the cohomology class in $H^1(J,\, {\mathcal O}_J)$
corresponding to the principal $\mathbb C$--bundle $Z$ coincides with the cohomology
class corresponding to the extension in \eqref{V}. Since the extension in \eqref{V}
is non-split, we conclude that the cohomology class in $H^1(J,\, {\mathcal O}_J)$
corresponding to the principal $\mathbb C$--bundle $Z$ is nonzero. Now the
proposition follows from Lemma \ref{lem3} and Lemma \ref{lem4}.
\end{proof}

\begin{proposition}\label{prop3}
The variety ${\mathcal M}_X$ does not admit any nonconstant algebraic function.
\end{proposition}

\begin{proof}
Let
\begin{equation}\label{P}
p\, :\,{\mathbb P}\,:=\, {\mathbb P}(V)\,\longrightarrow\, J
\end{equation}
be the projective bundle over $J$ associated to $V$ in \eqref{V}. Note that the
projection $\sigma$ in \eqref{V} defines a section of the projective bundle
${\mathbb P}$ in \eqref{P}. The image of this section will be denoted by $S$. We have
\begin{equation}\label{i}
{\mathbb P}\setminus S\,=\, Z\, ,
\end{equation}
where $Z$ is constructed in \eqref{Z}. Let ${\mathcal O}_{\mathbb P} (1)$ denote
the tautological quotient line bundle of
$p^*V$. For any $n\, \geq\,1$, the tensor product ${\mathcal O}_{\mathbb P}
(1)^{\otimes n}$ will be denoted by ${\mathcal O}_{\mathbb P} (n)$.

The restrictions of both ${\mathcal O}_{\mathbb P}(1)$ and ${\mathcal O}_{\mathbb P}(S)$
to any fiber of $p$ are of degree one. Therefore, from the see--saw theorem,
\cite[p. 51, Corollary 6]{Mu}, we know
that there is a line bundle $\xi$ on $J$ such that
\begin{equation}\label{P2}
{\mathcal O}_{\mathbb P} (1)\otimes p^*\xi\,=\, {\mathcal O}_{\mathbb P}(S)\, .
\end{equation}
The restriction of ${\mathcal O}_{\mathbb P} (1)$ to $S$ is the trivial line bundle
because the quotient line bundle in \eqref{V} is trivial. From the Poincar\'e
adjunction formula we know that the restriction of ${\mathcal O}_{\mathbb P}(S)$ to
$S$ is the normal bundle to $S$ (see \cite[p. 146]{GH} for
the Poincar\'e adjunction formula). The normal bundle to $S$ is clearly identified
with the restriction to $S$ of the relative tangent bundle for the projection $p$
in \eqref{P}. From \eqref{V} we know that the restriction to $S$ of the
relative tangent bundle for $p$ is $Hom({\mathcal O}_J\, ,{\mathcal O}_J)\,=\,
{\mathcal O}_J$. Since both ${\mathcal O}_{\mathbb P} (1)\vert_S$ and
${\mathcal O}_{\mathbb P}(S)\vert_S$ are trivial, from \eqref{P2} we conclude that
the line bundle $\xi$ is trivial. Hence we have
\begin{equation}\label{P3}
{\mathcal O}_{\mathbb P} (1)\,=\, {\mathcal O}_{\mathbb P}(S)\, .
\end{equation}
So, ${\mathcal O}_{\mathbb P}(n)\,=\, {\mathcal O}_{\mathbb P} (n\cdot S)$
for all $n\, \geq\, 1$. Therefore, the projection formula says that
$$
p_*{\mathcal O}_{\mathbb P}(n\cdot S)\,=\, p_* ({\mathcal O}_{\mathbb P} (n))\,=\,
\text{Sym}^n(V)\, .
$$
This implies that
\begin{equation}\label{e6}
H^0({\mathbb P},\, {\mathcal O}_{\mathbb P}(n\cdot S))\,=\,
H^0(J,\, f_*({\mathcal O}_{\mathbb P}(n\cdot S)))\,=\, H^0(J,\,\text{Sym}^n(V))\, .
\end{equation}
The vector bundle $\text{Sym}^n(V)$ is isomorphic to the vector bundle $F_{n+1}$
in \cite[p. 432, Theorem 5]{At1}. It is shown in \cite{At1} that
$$
\dim H^0(J,\, F_{n})\,=\, 1
$$
(this is proved in \cite[p. 430, Lemma 15]{At1}; see also the penultimate line
in \cite[p. 432]{At1}). Therefore, from \eqref{e6} we conclude that
$$
H^0({\mathbb P},\, {\mathcal O}_{\mathbb P}(n\cdot S))\,=\, \mathbb C\, \ \ \ \forall \
n \, >\, 0\, .
$$
Now from \eqref{i} it follows that
$$
H^0(Z,\, {\mathcal O}_{Z})\,=\, \mathbb C\, .
$$
In view of this, the proposition follows from Proposition \ref{prop2}.
\end{proof}

\section{Functions of moduli space of $G$--connections}

Let $G$ be a connected reductive affine algebraic group over $\mathbb C$.
Let ${\mathcal M}_X(G)$ denote the moduli space of topologically
trivial algebraic $G$--connections on
$X$, meaning pairs of the form $(E_G\, , D)$, where $E_G$ is an algebraic
principal $G$--bundle on $X$ such that $E_G$ is topologically trivial,
and $D$ is an algebraic connection on $E_G$.

\begin{theorem}\label{thm1}
The variety ${\mathcal M}_X(G)$ does not admit any nonconstant algebraic function.
\end{theorem}

\begin{proof}
We saw in Proposition \ref{prop3} that the variety ${\mathcal M}_X$
does not admit any nonconstant algebraic function. Therefore, for any positive
integer $d$, the Cartesian product $({\mathcal M}_X)^d$ does not admit any
nonconstant algebraic function.

Fix a maximal torus ${\mathcal T}\, \subset\, G$. Let ${\mathcal M}_X({\mathcal T})$
be the moduli space of ${\mathcal T}$--connections on $X$, meaning pairs of the
form $(E_{\mathcal T}\, ,D)$, where $E_{\mathcal T}$ is an algebraic principal
${\mathcal T}$--bundle on $X$ and $D$ is an algebraic connection on $E_{\mathcal T}$.
We note that since ${\mathcal T}$ is a product of copies of
the multiplicative group ${\mathbb C}^*\,=\, {\mathbb C}\setminus \{0\}$, if an
algebraic principal ${\mathcal T}$--bundle $E_{\mathcal T}$ admits an algebraic
connection, then $E_{\mathcal T}$ is topologically trivial. The inclusion of $\mathcal T$
in $G$ produces a morphism
$$
\psi\, :\, {\mathcal M}_X({\mathcal T})\,\longrightarrow\, {\mathcal M}_X(G)\, .
$$
This morphism $\psi$ is known to be surjective \cite[p. 1177, Theorem 4.1]{Th}.

Let $\delta$ denote the dimension of ${\mathcal T}$, so ${\mathcal T}$ is isomorphic
to $({\mathbb C}^*)^\delta$. Therefore, ${\mathcal M}_X({\mathcal T})$ is isomorphic to
the variety $({\mathcal M}_X)^\delta$. Since $({\mathcal M}_X)^\delta$ does not admit
any nonconstant algebraic function, and $\psi$ is surjective, we conclude that
${\mathcal M}_X(G)$ does not admit any nonconstant algebraic function.
\end{proof}

\subsection{The Betti moduli space}\label{se4.1}

We now recall the definition of the Betti moduli space $M^B_X(G)$ associated to the
pair $(X\, ,G)$ \cite{Si}. The identity element of $G$ will be denoted by $e$. We will
identify $\pi_1(X)$ with ${\mathbb Z}\bigoplus {\mathbb Z}$. Consider the morphism
$$
f\, :\, G\times G\,\longrightarrow\, G\, , ~\ (x\, ,y)\, \longmapsto\,
xyx^{-1}y^{-1}\, .
$$
Let ${\mathcal R}_X(G)$ denote the connected component, containing $(e\, ,e)$, of
$f^{-1}(e)$. It is an affine variety because $G$ is so. The simultaneous conjugation
action of $G$ on $G\times G$ preserves ${\mathcal R}_X(G)$. The geometric invariant
theoretic quotient ${\mathcal R}_X(G)/\!\!/G$ is the Betti moduli space $M^B_X(G)$.
It is an affine variety. The Riemann--Hilbert correspondence produces a biholomorphism
between ${\mathcal M}_X(G)$ and $M^B_X(G)$; it sends an algebraic connection to
the monodromy of the corresponding flat connection.

In contrast with Theorem \ref{thm1}, the variety ${\mathcal M}_X(G)$ has plenty
of holomorphic functions because it is biholomorphic to an affine variety.


\end{document}